\documentclass[11pt,reqno]{amsart}
\usepackage{amsmath, amsthm, amssymb}
\usepackage{mathrsfs}
\usepackage{array}
\usepackage{caption}

\evensidemargin \oddsidemargin
\newcommand{\mmod}[1]{\,\,(\text{\rm mod}\,\, #1)}

\def\bfx{{\mathbf x}}

\def\bfn{{\boldsymbol n}}

\def\bfalpha{{\boldsymbol \alpha}}

\def\bfnu{{\boldsymbol \nu}}

\newtheorem{thm}{Theorem}

\newtheorem{lem}{Lemma}

\newtheorem{conj}{Conjecture}

\newtheorem{prop}{Proposition}

\numberwithin{equation}{section} \numberwithin{thm}{section}
\numberwithin{lem}{section} \numberwithin{problem}{section}
\numberwithin{cor}{section}

\def\grm{{\mathfrak m}}\def\grM{{\mathfrak M}}\def\grN{{\mathfrak N}}\def\grn{{\mathfrak n}}\def\grP{{\mathfrak P}}

\begin{document}
\title{Uniform bounds in Waring's problem over some diagonal forms}
\author[Javier Pliego]{Javier Pliego}
\address{School of Mathematics, University of Bristol, University Walk, Clifton, Bristol BS8 1TW, United 
Kingdom.}

\address{Current address: Purdue Mathematical Science Building, 150 N University St, West Lafayette, IN 47907, United States of America.}

\email{jp17412@bristol.ac.uk}

\begin{abstract} We investigate the existence of representations of every large positive integer as a sum of $k$-th powers of integers represented as certain diagonal forms. In particular, we consider a family of diagonal forms and discuss the problem of giving a uniform upper bound over the family for the number of variables needed to have such representations.
\end{abstract}
\maketitle

\section{Introduction} 
Waring's problem (first resolved by Hilbert) asserts that for every $k\in\mathbb{N}$ there exists $s=s_{0}(k)$ such that all positive natural numbers can be written as a sum of $s$ positive integral $k$-th powers. Likewise, the problem of representing a sufficiently large natural number $n$ in the shape 
\begin{equation}\label{solu}n=x_{1}^{k}+\dots+x_{s}^{k},\end{equation} with $x_{i}\in \mathcal{S},$ where $\mathcal{S}$ is a given subset of the integers, has also been studied for particular cases.  However, little has been written about Waring's problem when one considers specific sparse sets, and apart from the set of prime numbers, not much can be found on the literature. It is then rare to encounter examples of sparse sets with a structure fundamentally different in nature in which Waring's problem along the lines of equation (\ref{solu}) is solvable. 

For a general set $\mathcal{A}\subset\mathbb{N}$, denote by $G_{\mathcal{A}}(k)$ the least positive integer $s$ such that for all sufficiently large natural numbers $n$, the equation (\ref{solu}) possesses a solution with $x_{i}\in\mathcal{A}.$ If such a number does not exist, define it as $\infty$. Let $\alpha>0$ and consider sets $\mathcal{A}_{\alpha}\subset\mathbb{N}$ well distributed on arithmetic progressions and with the property that $\lvert\mathcal{A}_{\alpha}\cap[1,N]\lvert\gg N^{\alpha}$. Then, on denoting $$W(k,\alpha)=\sup_{\mathcal{A}_{\alpha}}\big\{G_{\mathcal{A}_{\alpha}}(k)\big\},$$one would hope to have $W(k,\alpha)<\infty.$ In what follows we describe a particular family of sets satisfying the above property for which we expect the previous uniform bound to hold, but first we introduce, for convenience, some notation. For fixed $k,l,t\in\mathbb{N}$, let $\mathbf{x}=(x_{1},\ldots,x_{t})\in\mathbb{N}^{t}$, consider the function $T_{t}(\mathbf{x})=x_{1}^{l}+\ldots+x_{t}^{l}$ and take the set 
$$\mathcal{T}_{t}=\big\{T_{t}(\mathbf{x}):\ \mathbf{x}\in\mathbb{N}^{t}\big\}.$$ 

In this memoir we restrict our attention to the analysis of the solubility of (\ref{solu}) for the choice $\mathcal{S}=\mathcal{T}_{t}$ with $t$ lying in the following two regimes:

$(i)$ When $t=C(k)l$ for any fixed integer-valued function satisfying $C(k)\geq \textstyle{\frac{1}{2}}\log (k(k+1)).$ Work of Wooley \cite{Woo1} then yields the lower bound $\lvert \mathcal{T}_{t}\cap [1,N]\lvert \gg N^{1-\beta/k^{2}}$ for some constant $\beta>0$. The reader may notice that once we fix $k$, the above bound is uniform over the family of sets $\mathcal{T}_{t}$, whence in view of the preceding discussion we expect to have $G_{\mathcal{T}_{t}}(k)<G_{C}(k)$ for all $t$ in this regime with $G_{C}(k)$ being a constant depending on $k$ and $C$. As will be discussed afterwards, the lower bound for the cardinality of the sets available is not strong enough to prove such a statement, and we end up showing something weaker.

$(ii)$ When $t\geq \textstyle{\frac{l}{2}}\big(\log l+\log(k(k+1))+2\big)$ then work of Wooley \cite{Woo1} yields the stronger lower bound $\lvert \mathcal{T}_{t}\cap [1,N]\lvert \gg N^{1-\gamma/lk^{2}}$ for some absolute constant $\gamma>0$ at the cost of taking more variables. Were the sets $\mathcal{T}_{t}$ to have positive density, the argument would be considerably simplified and a pedestrian approach of the circle method would suffice. We use though the estimate available for the cardinality of these sets to derive a bound for $G_{\mathcal{T}_{t}}(k)$ that only depends on $k$. 

For the rest of the introduction we discuss each of the two regimes described above and provide some motivation underlying their choice. As experts will realise, an application of the Hardy-Littlewood method delivers the solubility of (\ref{solu}) for $\mathcal{S}=\mathcal{T}_{t}$ when $t=C(k)l$ and $s$ is large enough in terms of $k$ and $l$. We denote by $S_{C}(k,l)$ the minimum $s$ with such a property and consider
$$P_{C}(k)=\sup_{l\geq 2}\big\{S_{C}(k,l)\big\},$$ which does not necessarily have to be finite. We also define the constant
\begin{equation}\label{cono}\delta_{r}=\exp(1-2r/l)\end{equation}for each $r\in\mathbb{N}$ and note that then combining the corollary to Theorem 2.1 of Wooley \cite{Woo1} and a standard argument involving Cauchy's inequality one obtains the lower bound\begin{equation}\label{1.fi}\lvert \mathcal{T}_{t}\cap [1,N]\rvert \gg N^{1-\delta_{t}},\end{equation}where $\delta_{t}=\exp\big(1-2C(k)\big)$ just depends on $k$. As previously mentioned, the estimate (\ref{1.fi}) is uniform once we fix $k$, whence the discussion made above motivates the following conjecture.

\begin{conj}\label{con1}
Let $k\in\mathbb{N}.$ There exists a positive integer-valued function $C:\mathbb{N}\rightarrow\mathbb{N}$ such that $P_{C}(k)<\infty.$
\end{conj}
This conjecture seems to be out of reach with the methods available in the literature for any value of $k$. However, in this paper we make some progress by using an argument which permits us to prove a weaker version which we describe next after introducing first some notation. For $s\in\mathbb{N}$, the choice of $t$ described above and any $r\geq 0$, consider the equation
\begin{equation}\label{conje}
n=\sum_{i=1}^{s}T_{t}(\mathbf{x}_{i})^{k}+\sum_{i=1}^{r}x_{i}^{k},
\end{equation}
where $\mathbf{x}_{i}\in\mathbb{N}^{t}$ and $x_{i}\in\mathbb{N}.$ Let $S_{C}(k,l,r)$ denote the minimum number such that for $s\geq S_{C}(k,l,r),$ the equation (\ref{conje}) has a solution for all sufficiently large $n$ and take $$P_{C}(k,r)=\sup_{l\geq 2}\{S_{C}(k,l,r)\},$$which, as before, does not necessarily have to be finite. We define $R_{C}(k)$ to be the minimum $r\geq 0$ such that $P_{C}(k,r)$ is finite. After the preceding discussion we are now equipped to state the main theorem of the paper.
\begin{thm}\label{thm1.2}
Let $k\geq 2$ and consider any positive integer-valued function $C(k)$ with the property that $$C(k)\geq \max\big(4,\textstyle{\frac{1}{2}}\log (k(k+1))+3/2\big).$$ Then one has the bound
$$R_{C}(k)\leq 4,$$ and for every $r\geq 4$ one finds that $P_{C}(k,r)\leq k^{2}+O(k).$ Moreover, $R_{C}(2)\leq 2$.
\end{thm}
We should emphasize that one could obtain the more precise bound $P_{C}(k,r)\leq k(k+1)$ by introducing suitable weights in the exponential sums that we make use of and exploiting the information provided by such sums on the major arc analysis. We have omitted providing that discussion to make the exposition simpler. The reader might as well want to observe that $R_{C}(k)=0$ is equivalent to Conjecture \ref{con1}, whence the statement containing the relevant information in the above theorem is the upper bound on $R_{C}(k)$. 

Let $G(k)=G_{\mathbb{N}}(k)$ be the smallest number such that for all $s\geq G(k)$, every large enough natural number can be written as a sum of $s$ positive integral $k$-th powers. Vinogradov \cite{Vi}, Karatsuba \cite{K} and Vaughan \cite{Vau3} made progress to achieve upper bounds for $G(k)$, the best current one for large $k$ being \begin{equation}\label{G(k)}G(k)\leq k\Bigg(\log k+\log\log k+2+O\Big(\frac{\log\log k}{\log k}\Big)\Bigg)\end{equation} due to Wooley \cite{Woo2}. Note that as a consequence of this bound one trivially has $$R_{C}(k)\leq k\big(\log k+\log\log k+O(1)\big).$$ The reader then might want to observe that Theorem \ref{thm1.2} improves this bound substantially. It is also worth noting that if Conjecture \ref{con1} were true for any fixed $k$, there would exist some $s=s(k)$ with the property that for any $l\geq 2$, every sufficiently big enough integer $n$ would have a representation of the shape
\begin{equation*}
n=\sum_{i=1}^{s}T_{t}(\mathbf{x}_{i})^{k}
\end{equation*}
with $\mathbf{x}_{i}\in\mathbb{N}^{t}$. Observe that the right side of the above equation would consist of sums of $Cl$ positive integral $l$-th powers gathered in groups and raised to the power $k$ for some constant $C=C(k)>0$ depending on $k$. This problem seems then even harder than the problem of proving that every sufficiently large integer can be written as the sum of $Cl$ positive integral $l$-th powers, which would be a big breakthrough in view of (\ref{G(k)}).

Before describing the other regime for $t$ analysed in the memoir, we note that as a consequence of the aforementioned work on $G(k)$, it follows that whenever $t\geq t_{0}(l)$ with \begin{equation}\label{tfor}t_{0}(l)=\frac{l}{2}\big(\log l+\log\log l+2+o(1)\big)\end{equation} then $\mathcal{T}_{t}$ has positive density, which greatly simplifies things (see, for example Br\"udern, Kawada and Wooley \cite[Theorem 1.5]{BKW}). With the current state of knowledge, this turns out to be the threshold for which we can guarantee to have a lower bound of the shape $\lvert\mathcal{T}_{t}\cap [1,N]\rvert\gg N^{1-\varepsilon}.$ Therefore, for fixed $k$ and $l$ large enough, the cardinality of the sets $\mathcal{A}_{l}=\mathcal{T}_{\xi_{0}(k,l)}\cap [1,N]$ with $\xi_{0}(k,l)=\lceil l/2\big(\log l+\log\big(k(k+1)\big)+2\big)\rceil$ is not known to satisfy $\lvert \mathcal{A}_{l}\rvert\gg N^{1-\varepsilon},$ the best lower bound known being
$$\lvert \mathcal{A}_{l}\rvert\gg N^{1-1/k(k+1)le},$$which is a consequence of (\ref{1.fi}).
\begin{thm}\label{thm1.3}
Let $k,l\geq 2$ and take $\xi\geq \xi_{0}(k,l)$ and $s\geq s_{0}(k)$ with $s_{0}(k)=k^{2}+O(k).$ Then every sufficiently large $n$ can be represented as
$$n=\sum_{i=1}^{s}x_{i}^{k},$$
where $x_{i}\in\mathcal{T}_{\xi}$.
\end{thm}
The reader may want to observe that even if the sets $\mathcal{T}_{\xi}$ are not known to satisfy the estimate $\lvert\mathcal{T}_{\xi}\cap [1,N]\rvert\gg N^{1-\varepsilon}$, the bound on the number $s$ of variables needed does not depend on $l$. This suggests that one should search for ideas which don't just make use of the polynomial structure of the sets $\mathcal{T}_{\xi}$ in order to prove such result. Experts in the area may also notice that one could prove a weaker version of the theorem by combining Corollary 1.4 of Wooley \cite{Woo4} with a pointwise bound over the minor arcs derived from Lemma 5.4 of Vaughan \cite{Vau}. This strategy though would entail the restriction $s\geq (3/2)k^{2}+O(k)$. We instead use a similar idea than the one we employ for the minor arc treatment in the proof of Theorem \ref{thm1.2} that avoids relying on such pointwise bounds and enables us to win $k(k+1)/2$ variables. It is worth mentioning that one could also introduce suitable weights in the exponential sum that we make use of to obtain a more precise error term in the expression for $s_{0}(k)$.

Back to equation (\ref{solu}), the case when the set $\mathcal{S}$ is taken to be the prime numbers has been of interest to many mathematicians. Among others, Hua first (\cite{Hua}, \cite{Hua2}) and then Thanigasalam (\cite{Tha1}, \cite{Tha2}),  Kumchev \cite{Kum},  Kawada and Wooley \cite{KW}, and Kumchev and Wooley (\cite{KW1}, \cite{KW2}) have worked to give upper bounds for $H(k)$, where $H(k)$ is defined as the minimum number such that for every $s\geq H(k)$, the equation $$n=p_{1}^{k}+\dots+p_{s}^{k}$$ has a solution for all sufficiently large $n$ with the property that $n\equiv s\pmod{K(k)},$ where $K(k)$ is a constant defined in terms of $K(k)$ to ensure appropiate local solubility conditions (see \cite{KW1} for a more precise definition of $K(k)$). We note that the best current bound for large $k$ is $H(k)\leq (4k-2)\log k-(2\log 2-1)k-3$ due to Kumchev and Wooley \cite{KW2}.

At the same time, some authors have been trying to find sparse sets with minimum density such that the problem of representing every sufficiently large positive number as a sum of $k$-th powers of elements of the set is still soluble with strong upper bounds for the number of variables needed. This other approach in Waring's problem has been studied by Nathanson
\cite{Nat}, where he used the probabilistic method to prove that for all $s\geq G(k)+1,$ there exist sets $A$ with $\text{card}(A\cap [1,N])\sim c N^{1-1/s+\varepsilon}$  such that (\ref{solu}) is soluble on $A$. The result was partially improved by Vu \cite{Vu}, when he showed under the condition $s\geq k^{4}8^{k}$ the existence of a set $A\subset\mathbb{N}$ with $\text{card}(A\cap[1,N])=\Theta\big(N^{k/s}(\log N)^{1/s}\big)$ such that $R_{A}(n)\asymp \log n,$ where $R_{A}(n)$ denotes the number of solutions of (\ref{solu}) with the variables lying in $A$. Later on, Wooley \cite{Woo3} proved the same result for $s\geq T(k)+2$, where $T(k)$ is bounded above by an explicit version of the right-hand side of (\ref{G(k)}).  However, though the size of the sets is near optimal, the arguments used by the authors are probabilistic, so they don't give a description of those sets. Therefore, the approach of this paper might be the first one in which by giving an explicit family of sets with similar density, one tries to find a uniform bound for the number of variables needed to solve (\ref{solu}), as discussed at the beginning of the introduction.

Theorems \ref{thm1.2} and \ref{thm1.3} are proved via the circle method, and the exposition is organised as follows. We bound a mean value via restriction estimates in Section \ref{sec2} and we expose the key argument which permits us to estimate mean values of a suitable exponential sum over the minor arcs uniformly on $l$. Section \ref{sec3} is devoted to a brief study of the singular series. In Section \ref{sec4} we approximate the generating function for the problem over the major arcs. We give an asymptotic formula for the integral of the product of some exponential sums over the major arcs in Section \ref{sec5} and we use it to complete the proof of Theorem \ref{thm1.2}. We have included a small note in Section \ref{sec6} that deals with the case $k=2$. In Section \ref{sec7} we slightly modify the exponential sum taken in Section \ref{sec2} and use a similar argument to obtain a suitable estimate for the contribution of the minor arcs in the setting of Theorem \ref{thm1.3}. We combine such work with a standard major arc analysis to prove the theorem.

For the rest of the paper, we fix positive integers $l\geq 2$ and $k\geq 2$. For the sake of simplicity concerning local solubility, we assume that $t\geq 4l$, though most of the results throughout the paper don't require this restriction. For ease of notation we also write $T(\mathbf{x})$ instead of $T_{t}(\mathbf{x})$. The main objective of Theorem \ref{thm1.2} is to prove a non-trivial uniformity bound for $S_{C}(k,l,r)$, and thus we just focus our attention in large values of $l$ in terms of $k$. As mentioned above, even if we provide explicit bounds for $P_{C}(k,r)$ and $C(k)$, the relevant part of the result is the estimate on $R_{C}(k)$. For such purposes, we haven't included an investigation of the behaviour of $P_{C}(k,r)$ and $C(k)$ for small $k$.

As usual in analytic number theory, we denote $e^{2\pi i z}$ by $e(z)$, and for every $q\in\mathbb{N}$, we put $e_{q}(z)=e^{2\pi i z/q}$. When we write $a\leq \mathbf{x}\leq b$ for a vector $\mathbf{x}=(x_{1},\ldots,x_{s})\in\mathbb{R}^{s}$ we will mean that $a\leq x_{i}\leq b$ for all $1\leq i\leq s.$ We denote $\mathbf{x}\equiv \mathbf{y}\pmod{q}$ when $x_{i}\equiv y_{i}\pmod{q}$ for all $1\leq i\leq s.$ We write $p^{r}|| n$ to denote that $p^{r}| n$ but $p^{r+1}\nmid n.$ Whenever $\varepsilon$ appears in any bound, it will mean that the bound holds for every $\varepsilon>0$, though the implicit constant then may depend on $\varepsilon$. We adopt the convention that when we write $\delta$ in the computations we mean that there exists a positive constant such that the bound holds. We use $\ll$ and $\gg$ to denote Vinogradov's notation. 

\emph{Acknowledgements}: The author's work was supported in part by a European Research Council Advanced
Grant under the European Union's Horizon 2020 research and innovation programme via grant agreement No. 695223 during his studies at the University of Bristol. It was completed while the author was visiting Purdue University under Trevor Wooley's supervision. The author would like to thank him for his guidance and helpful comments, and both the University of Bristol and Purdue University for their support and hospitality.

\section{Minor arc estimate}\label{sec2}
We will begin by displaying an upper bound for mean values of an exponential sum which will be of later use in the analysis of the minor arcs in the setting of both Theorems \ref{thm1.2} and \ref{thm1.3}. This will be a straightfoward consequence of the work of Wooley \cite{Woo4} on Vinogradov's mean value theorem with weights. 
Let $r\in \mathbb{N}$, let $Y>0$ be a real parameter and consider the set
\begin{equation}\label{Sr}\mathcal{S}_{r}(Y)=\Big\{x_{1}^{l}+\ldots+x_{r}^{l}:\ \ \ x_{i}\in\mathcal{A}(Y,Y^{\eta}),\ \ (1\leq i\leq r) \Big\},\end{equation} where
\begin{equation*}\mathcal{A}(Y,R)=\{n\in [1,Y]\cap \mathbb{N}: p\mid n\text{ and $p$ prime}\Rightarrow p\leq R\}\end{equation*} and $\eta$ is a sufficiently small but positive parameter. Note that then the corollary to Theorem 2.1 of  Wooley \cite{Woo1} and a routine argument using Cauchy's inequality yield \begin{equation}\label{1.1}\lvert \mathcal{S}_{r}(Y)\rvert \gg Y^{l-l\delta_{r}},\end{equation}where $\delta_{r}$ was defined in (\ref{cono}). In order to make further progress we need to introduce first some notation. Let $n$ be a positive integer and take $X=n^{1/k}$ and $P=X^{1/l}.$ Define for $\alpha\in [0,1)$ and $\bfalpha\in[0,1)^{k}$ the exponential sums
\begin{equation}\label{falphap}f\big(\alpha,\mathcal{S}_{r}(Y)\big)=\sum_{x\in \mathcal{S}_{r}(Y)}e(\alpha x^{k}),\  \ \ \ \ \ \ \ \ f\big(\bfalpha,\mathcal{S}_{r}(Y)\big)=\sum_{x\in \mathcal{S}_{r}(Y)}e(\alpha_{1}x+\ldots+\alpha_{k}x^{k}).\end{equation} For future purposes in the analysis, we consider the mean value
\begin{equation}\label{ec2.9}J_{s,r}^{(k)}(Y)=\int_{[0,1)^{k}}\big\lvert f\big(\bfalpha,\mathcal{S}_{r}(Y)\big)\big\rvert^{2s} d\bfalpha,\end{equation} which by orthogonality counts the solutions to the system $$x_{1}^{j}+\ldots+x_{s}^{j}=x_{s+1}^{j}+\ldots+x_{2s}^{j}\ \ \ \ \ \ \ \ \ \ (1\leq j\leq k),$$ where $x_{i}\in\mathcal{S}_{r}(Y).$

\begin{prop}\label{prop1}
Let $s\geq k(k+1)/2.$ Then one has that
\begin{equation*}J_{s,r}^{(k)}(Y) \ll \lvert \mathcal{S}_{r}(Y)\rvert ^{2s}Y^{-lk(k+1)/2+l\Delta_{r}+\varepsilon},\end{equation*}
where $\Delta_{r}=\delta_{r}k(k+1)/2$.
\end{prop}
\begin{proof}
Define the weights $a_{x}=1$ if $x\in \mathcal{S}_{r}(Y)$ and $0$ otherwise. Then, we can rewrite $f(\bfalpha,\mathcal{S}_{r}(Y))$ as
$$f(\bfalpha,\mathcal{S}_{r}(Y))=\sum_{x\leq rY^{l}}a_{x}e(\alpha_{1}x+\ldots+\alpha_{k}x^{k}).$$ Therefore, combining Corollary 1.4 of Wooley \cite{Woo4} with (\ref{1.1}) and the triangle inequality we obtain $$J_{s,r}^{(k)}(Y)\ll \lvert\mathcal{S}_{r}(Y)\rvert ^{2s}\lvert \mathcal{S}_{r}(Y)\rvert^{-k(k+1)/2}Y^{\varepsilon}\ll\lvert \mathcal{S}_{r}(Y)\rvert ^{2s}Y^{-lk(k+1)/2+l\Delta_{r}+\varepsilon},$$ which yields the desired result. 
\end{proof}

Before introducing the main ingredients for the minor arc analysis, we recall from the introduction that whenever $t\geq t_{0}(l)$, where $t_{0}(l)$ was defined in (\ref{tfor}), then $\mathcal{T}_{t}$ has positive density. We deliberately avoid this situation by considering $l$ sufficiently large in terms of $k$. The difficulty of the Conjecture \ref{con1} then lies on the fact that $\mathcal{T}_{t}$ is not known to have positive density, and the best lower bounds available on the cardinality of the set are not strong enough. Moreover, any approach making use of the fact that $T_{t}(\mathbf{x})^{k}$ is a polynomial of degree $kl$ and applying a Weyl estimate for the corresponding exponential sum or any multivariable version of Vinogradov's Mean Value Theorem (see Theorem 1.3 and Theorem 2.1 of \cite{PPW}) would entail a restriction in the number of variables that would depend on the degree of the polynomial. 

We make though some progress by obtaining a uniform bound in $l$ of a suitable exponential sum over the minor arcs. Our argument here is motivated by the treatment of Vaughan \cite[Chapter 5]{Vau}, and it requires the estimate in Proposition \ref{prop1}. For such purposes, we introduce first some notation. Consider a positive integer-valued function $C(k)$ and set \begin{equation*}\label{tes}t=C(k)l,\ \ \ \ \ \ \ \ \ t_{1}=C(k)l-l.\end{equation*} We take the parameter \begin{equation}\label{varphi}\varphi_{k,t_{1}}=1-k(k+1)\delta_{t_{1}}/2-e^{-1},\end{equation} which for the sake of concision will be denoted by $\varphi_{k}$. Observe that whenever $C(k)$ satisfies the lower bound included in the hypothesis of Theorem \ref{thm1.2} then one has $\varphi_{k}\geq 1/2-e^{-1}>0.$ We also define the constants
\begin{equation}\label{C1}C_{1}=\big(k(k+1)2^{k+1}t_{1}^{k}\big)^{-1/lk},\ \ \ \ \ \ \ C_{2}=\min\Big((2lk)^{-1/l},\big(k(k+1)2^{k+1}l^{k}\big)^{-1/lk}\Big).\end{equation}
and the natural numbers
\begin{equation}\label{Pt}P_{1}=\lfloor C_{1}P\rfloor,\ \ \ \ P_{2}=\lfloor C_{2}P\rfloor.\end{equation}For ease of notation, we denote $\mathcal{S}_{t_{1}}(P_{1}),$ and $\mathcal{S}_{l}(P_{2})$ by $\mathcal{S}_{1},$ and $\mathcal{S}_{2}$ respectively. It is then convenient to define, for $m\in\mathcal{S}_{2},$ the exponential sums
\begin{equation*}f_{m}(\alpha)=\sum_{x\in\mathcal{S}_{1}}e\big(\alpha(x+m)^{k}\big)\ \ \ \ \ \ \ \  \text{and}\ \ \ \ \ \ \ \mathcal{F}(\alpha)=\sum_{m\in \mathcal{S}_{2}}f_{m}(\alpha).\end{equation*}
In order to make further progress we make a Hardy-Littlewood dissection. When $1\leq Q\leq X$ we define the major arcs $\grM(Q)$ to be the union of 
\begin{equation}\label{grM}\grM(a,q)=\Big\{ \alpha\in [0,1): \lvert \alpha-a/q\rvert \leq \frac{Q}{qn}\Big\}\end{equation} with $0\leq a\leq q\leq Q$ and $(a,q)=1$. For the sake of brevity we write
$$\grM=\grM(X),\ \ \ \ \ \ \ \ \grN=\grM(P^{1/2}),\ \ \ \ \ \ \ \ \grP=\grM(\log P)$$and we take $\grm=[0,1)\setminus \grM$ and $\grn=[0,1)\setminus \grN$ to be the minor arcs. 

\begin{prop}\label{prop22}
Let $\alpha\in \grm$. Then one has
$$\mathcal{F}(\alpha)\ll |\mathcal{S}_{1}||\mathcal{S}_{2}|X^{-\varphi_{k-1}/k(k-1)+\varepsilon},$$ where $\varphi_{k-1}$ was defined in (\ref{varphi}). 
Moreover, for $s\geq k(k+1)/2$ we obtain the mean value estimate
\begin{equation*}\int_{\grm}\lvert\mathcal{F}(\alpha)\rvert^{2s}d\alpha\ll |\mathcal{S}_{1}|^{2s}|\mathcal{S}_{2}|^{2s}X^{-k-\varphi_{k}+\varepsilon}.\end{equation*}
\end{prop}
As experts will realise throughout the proof, one could obtain a similar result for the analogous Vinogradov generating function by using ideas of the proof of Theorem 5.2 of Vaughan \cite{Vau}. We have ommited such analysis for the clarity of the exposition.
\begin{proof}
For every $m\in\mathcal{S}_{2}$ consider $\gamma(m)=(\gamma_{1}(m),\ldots,\gamma_{k-1}(m))$, where the entries taken are
\begin{equation}\label{ec1.33}\gamma_{j}(m)=\alpha\binom{k}{j}m^{k-j},\ \ \ \ \ \ (1\leq j\leq k-1).\end{equation}
Observe that then for every $x\in\mathcal{S}_{1}$ one obtains the relation
\begin{equation*} \alpha(x+m)^{k}=\bfnu^{(k-1)}(x)\cdot \gamma(m)+\alpha x^{k}+\alpha m^{k},\end{equation*} where we adopted the notation $\bfnu^{(k-1)}(x)=(x,\ldots,x^{k-1}).$ It is also convenient to define for $s\in\mathbb{N}$ the set of $(k-1)$-tuples of natural numbers
$$\mathcal{N}=\Big\{(n_{1},\ldots,n_{k-1}):\ \ 1\leq n_{i}\leq sX^{i},\ \ (1\leq i\leq k-1)\Big\}.$$
By using the above relation we find that
\begin{equation}\label{ec2.8}\sum_{m\in\mathcal{S}_{2}}\lvert f_{m}(\alpha)\rvert^{2s}=\sum_{m\in\mathcal{S}_{2}}\Big\lvert\sum_{\bfn\in\mathcal{N}}a(\bfn)e\big(\bfn\cdot\gamma(m)\big)\Big\rvert^{2},\end{equation}
where on denoting
$$\mathcal{X}(\bfn)=\Big\{\mathbf{x}\in\mathcal{S}_{1}^{s}:\ x_{1}^{i}+\ldots+x_{s}^{i}=n_{i}, \ \ \ (1\leq i\leq k-1)\Big\}$$ the coefficient $a(\bfn)$ is defined as 
\begin{equation}\label{a(n}a(\bfn)=\sum_{\mathbf{x}\in \mathcal{X}(\bfn)}e\big(\alpha(x_{1}^{k}+\ldots+x_{s}^{k})\big).\end{equation}

We devote the rest of the proof to apply a version of the large sieve inequality (Lemma 5.3 of Vaughan \cite{Vau}) to the right side of (\ref{ec2.8}). For such purpose, we shall consider the spacing modulo $1$ of $\{\gamma(m)\}_{m}.$ Take $x,y\in\mathcal{S}_{2}$ with $x\neq y$. Note that in view of (\ref{C1}) then one has\begin{equation}\label{z}x,y\leq X/2k.\end{equation} Observe that applying Dirichlet's approximation to each $\alpha\in\grm$ we obtain $a\in\mathbb{Z}$ and $q\in\mathbb{N}$ with $(a,q)=1$ such that $0\leq a\leq q$ and $$\lvert \alpha-a/q\rvert \leq q^{-1}X^{1-k},$$ where $X<q\leq X^{k-1}.$ Note as well that by the choice of $\gamma_{j}(m)$ in (\ref{ec1.33}) we find that
$$\|k\alpha(x-y)\|=\|\gamma_{k-1}(x)-\gamma_{k-1}(y)\|.$$
Then by the above discussion we obtain the lower bound
$$\|k\alpha(x-y)\|\geq \|ka(x-y)/q\|-(2q)^{-1}X^{2-k}.$$
Note that the only instance in which the first term on the right-hand side of the above equation can be $0$ is when $y=x+nq(q,k)^{-1}$ for some $n\in\mathbb{N}$ with $n\neq 0$. However, $q(q,k)^{-1}\geq X/k,$ which would contradict (\ref{z}). Therefore, by the preceding discussion we get
$$\|k\alpha(x-y)\|\geq (2q)^{-1},$$which delivers
$\| \gamma_{k-1}(x)-\gamma_{k-1}(y)\|\gg X^{-k+1}$ and provides the spacing condition that we were seeking to prove.

Applying Lemma 5.3 of Vaughan \cite{Vau} to (\ref{ec2.8}) we obtain the upper bound
\begin{equation}\label{minstr}\sum_{m\in\mathcal{S}_{2}}\lvert f_{m}(\alpha)\rvert^{2s}\ll X^{k(k-1)/2}\sum_{\bfn\in\mathcal{N}}|a(\bfn)|^{2}.\end{equation} Note first that by bounding the coefficients $a(\bfn)$ trivially one gets
$$\sum_{m\in\mathcal{S}_{2}}\lvert f_{m}(\alpha)\rvert^{2s}\ll X^{k(k-1)/2}J_{s,t_{1}}^{(k-1)}(P_{1}),$$
where $J_{s,t_{1}}^{(k-1)}(P_{1})$ was defined in (\ref{ec2.9}). Then combining the above equation with an application of Cauchy's inequality we obtain
\begin{equation}\label{Falphat}\mathcal{F}(\alpha)^{2s}\ll |\mathcal{S}_{2}|^{2s-1}\sum_{m\in\mathcal{S}_{2}}\lvert f_{m}(\alpha)\rvert^{2s} \ll |\mathcal{S}_{2}|^{2s-1}X^{k(k-1)/2}J_{s,t_{1}}^{(k-1)}(P_{1}),\end{equation} and hence for $s\geq k(k-1)/2$ then Proposition \ref{prop1} delivers
$$\mathcal{F}(\alpha)\ll |\mathcal{S}_{1}||\mathcal{S}_{2}|\big(X^{\delta_{t_{1}}k(k-1)/2}|\mathcal{S}_{2}|^{-1}\big)^{1/2s}X^{\varepsilon}.$$ Therefore, fixing $s=k(k-1)/2$ and recalling (\ref{1.1}) one gets
\begin{align*}\mathcal{F}(\alpha)&
\ll|\mathcal{S}_{1}||\mathcal{S}_{2}|X^{-\varphi_{k-1}/k(k-1)+\varepsilon},
\end{align*} where $\varphi_{k-1}$ was defined in (\ref{varphi}). 

For the second claim of the proposition we combine (\ref{minstr}) and Cauchy's inequality in the same way as in (\ref{Falphat}) and we integrate over $\grm$ to get
$$\int_{\grm}\lvert\mathcal{F}(\alpha)\rvert^{2s}d\alpha\ll |\mathcal{S}_{2}|^{2s-1}X^{k(k-1)/2}J_{s,t_{1}}^{(k)}(P_{1}).$$An application of Proposition \ref{prop1} and the estimate (\ref{1.1}) to the above line then yields, for $s\geq k(k+1)/2$, the bound
 $$\int_{\grm}\lvert\mathcal{F}(\alpha)\rvert^{2s}d\alpha\ll |\mathcal{S}_{1}|^{2s}||\mathcal{S}_{2}|^{2s}X^{-k-\varphi_{k}+\varepsilon},$$from where the second part of the proposition follows.
\end{proof}
Observe that the argument just makes use of the fact that $\mathcal{T}_{t}=\mathcal{T}_{t_{1}}+\mathcal{T}_{l}$, where 
$\lvert\mathcal{T}_{t_{1}}\cap [1,N]\rvert$ and $\lvert\mathcal{T}_{l}\cap [1,N]\rvert$ are appropriately large. Therefore, it could also be applied to other problems for sets with a similar property that don't necessarily have a polynomial structure. We conclude the investigation of the minor arcs by applying Weyl differencing to derive a bound for the exponential sum \begin{equation}\label{fte}f(\alpha)=\sum_{T_{t}(\mathbf{z})\leq P^{l}}e(\alpha T_{t}(\mathbf{z})^{k}),\end{equation} where in the above sum $z\in\mathbb{N}^{t}$. For ease of notation, we avoid writing the dependance on $t$. Note that then one can rewrite $f(\alpha)$ as 
\begin{equation}\label{f(alpha)}
f(\alpha)=\sum_{\substack{\mathbf{x}\in\mathbb{N}^{t-1}\\ P_{\mathbf{x}}\geq 1}}f_{\mathbf{x}}(\alpha),\ \ \ \ \ \ \ \text{where}\ f_{\mathbf{x}}(\alpha)=\sum_{1\leq x\leq P_{\mathbf{x}}}e\big(\alpha T(\mathbf{x},x)^{k}\big)\end{equation} and where we took the parameter $P_{\mathbf{x}}=\big(P^{l}-T_{t-1}(\mathbf{x})\big)^{1/l}.$
\begin{lem}\label{lem2} Let $\alpha\in [0,1)$ and suppose that there exist $a\in\mathbb{Z}$ and $q\in\mathbb{N}$ such that $(a,q)=1$ and $\lvert \alpha-a/q\rvert\leq q^{-2}.$ Then
\begin{equation*} f(\alpha) \ll P^{t+\varepsilon}(q^{-1}+P^{-1}+qP^{-kl})^{2^{1-kl}}.
\end{equation*}
\end{lem}
\begin{proof}
Observe that the polynomial $T(\bfx,x)$ is monic and of degree $kl$ on $x$. Note that the implicit constant in Weyl's inequality (Vaughan \cite[Lemma 2.4]{Vau}) does not depend on the coefficients which are not the leading one. Therefore, an application of such inequality to $f_{\bfx}(\alpha)$ delivers
$$f_{\bfx}(\alpha)\ll P_{\mathbf{x}}^{1+\varepsilon}\big(q^{-1}+P_{\mathbf{x}}^{-1}+qP_{\mathbf{x}}^{-kl}\big)^{2^{1-kl}},$$ which yields the above estimate by combining the bound $P_{\mathbf{x}}\leq P$ and (\ref{f(alpha)}).
\end{proof}

\section{Singular series}\label{sec3}
Throughout this section we will always assume that $t\geq 2l.$ Define for $a\in\mathbb{Z}$ and $q\in\mathbb{N}$ with $(a,q)=1$ the complete exponential sums
\begin{equation}\label{S(q,a)}S(q,a)=\sum_{1\leq\mathbf{r}\leq q}e_{q}\big(aT(\mathbf{r})^{k}\big)\ \ \ \ \ \text{and}\ \ \ \ \ \ S_{k}(q,a)=\sum_{r=1}^{q}e_{q}(ar^{k}).\end{equation}Note that by orthogonality we can express $S(q,a)$ as
$$S(q,a)=q^{-1}\sum_{u=1}^{q}S_{l}(q,u)^{t}S_{k}(q,a,-u),\ \ \ \ \ \ \ \ \text{where}\ \ S_{k}(q,a,-u)=\sum_{r=1}^{q}e_{q}(ar^{k}-ur).$$ Because of the quasi-multiplicative structure of the exponential sum, we will focus on the case $q=p^{h}$, where $p$ is a prime number. Then, with the above notation one has that
\begin{equation}\label{sing}S(p^{h},a)=p^{(t-1)h}S_{k}(p^{h},a)+E(p^{h},a),\end{equation}
where
$$E(p^{h},a)=p^{-h}\sum_{u=1}^{ p^{h}-1}S_{l}(p^{h},u)^{t}S_{k}(p^{h},a,-u).$$
We estimate $E(p^{h},a)$ by using classical estimates for the sums $S_{l}(p^{h},u)$ and $S_{k}(p^{h},a,-u)$. Applying then Theorems 4.2 and 7.1 of Vaughan \cite{Vau} we obtain
\begin{align}\label{ec3.2}E(p^{h},a)\ll p^{-h/k+\varepsilon}\sum_{u=1}^{p^{h}-1}p^{th(1-1/l)}(u,p^{h})^{t/l}&
\ll p^{h(t-1/k-t/l+1)+\varepsilon}\sum_{d=0}^{h-1}p^{d(t/l-1)}\nonumber
\\
&\ll p^{ht-h/k-(t/l-1)+\varepsilon}.
\end{align}
In order to provide more explicit bounds for the exponential sum $S(q,a)$ it is convenient to introduce first the multiplicative function $w_{k}(q)$ defined as 
$$w_{k}(p^{uk+v})=p^{-u-1}\ \ \ \ \text{when $u\geq 0$ and $2\leq v\leq k$},$$
\begin{equation*}w_{k}(p^{uk+v})=kp^{-u-1/2}\ \ \ \ \text{when $u\geq 0$ and $v=1$}.\end{equation*} 
Then, by Lemma 3 of Vaughan \cite{Vau1} we obtain \begin{equation}\label{putill}q^{-1}\lvert S_{k}(q,a)\rvert\ll w_{k}(q),\end{equation} 
whence combining (\ref{sing}) and (\ref{ec3.2}) with the quasi-multiplicative property of $S(q,a)$ we get \begin{equation}\label{piu}q^{-t}\lvert S(q,a)\rvert \ll w_{k}(q).\end{equation}
For future purposes in the memoir, we note that by applying the definition of $w_{k}(q)$ and multiplicativity then we obtain for $Q>0$ and $s\geq \max(4,k+1)$ the bounds
\begin{equation}\label{pri}\sum_{q\leq Q}w_{k}(q)^{2}\leq \prod_{p\leq Q}\big(1+C/p\big)\ll Q^{\varepsilon},\ \ \ \ \ \ \ \ \ \sum_{q\leq Q}qw_{k}(q)^{s}\leq \prod_{p\leq Q}\big(1+C/p\big)\ll Q^{\varepsilon}.
\end{equation}

Before defining the singular series, we need to consider first the exponential sum
\begin{equation}\label{WW}W(q,a)=\sum_{\substack{r=1\\ (r,q)=1}}^{q}e_{q}(a r^{k}),\end{equation} where $a\in\mathbb{Z}$ and $q\in\mathbb{N}$ with $(a,q)=1.$ Here the reader might want to observe that one can express the sum $W(p^{h},a)$ in terms of $S_{k}(p^{h},a)$ and $S_{k}(p^{h-k},a)$, and hence one can deduce the estimate 
\begin{equation}\label{varp}\varphi(q)^{-1}\lvert W(q,a)\rvert\ll w_{k}(q)\end{equation} by just applying multiplicativity and the bound (\ref{putill}). In order to make further progress, it is worth defining
\begin{equation*}S_{n}(q)=\sum_{\substack{a=1\\ (a,q)=1}}^{q}\big(q^{-t}S(q,a)\big)^{s}\big(q^{-1}S_{k}(q,a)\big)^{2}\big(\varphi(q)^{-1}W(q,a)\big)^{2}e_{q}(-an).\end{equation*} 
We also consider for convenience the series
\begin{equation}\label{sigma}\mathfrak{S}(n)=\sum_{q=0}^{\infty}S_{n}(q),\ \ \ \ \  \ \ \ \ \ \ \sigma(p)=\displaystyle\sum_{h=0}^{\infty} S_{n}(p^{h}),\end{equation} where $p$ is a prime number. We can provide a more arithmetic description of $\sigma(p)$ by defining the set 
$$\mathcal{M}_{n}(p^{h})=\Big\{(\mathbf{y},\mathbf{X})\in [1,p^{h}]^{4}\times [1,p^{h}]^{st}: \ p\nmid y_{1}y_{2},\ n\equiv \sum_{i=1}^{4}y_{i}^{k}+\sum_{i=1}^{s}T(\bfx_{i})^{k}\pmod{p^{h}}\Big\}$$
and considering the counting function $M_{n}(p^{h})=\lvert\mathcal{M}_{n}(p^{h})\rvert$. For each prime $p$ take $\tau\geq 0$ such that $p^{\tau}\| k$ and 
\begin{equation}\label{gama}\gamma=\begin{cases}\tau +1,&\text{when $p>2$ or when $p=2$ and $\tau=0$},\\ \tau+2,& \text{when $p=2$ and $\tau>0$}.\end{cases}\end{equation}

\begin{lem}\label{lem4.3}
Suppose that $s+3\geq \frac{p}{p-1}\big(k,p^{\tau}(p-1)\big)$ when $\gamma=\tau+1$, that $s+3\geq 2^{\tau+2}$ when $\gamma=\tau+2$ and $k>2$, and that $s\geq 2$ when $p=k=2.$ Suppose as well that $t\geq 4l.$ Then one has $M_{n}(p^{\gamma})>0.$
\end{lem}

\begin{proof}
It is worth noting first that Lemma 2.15 of Vaughan \cite{Vau} implies that $$T(\mathbf{x})=x_{1}^{l}+\dots+x_{t}^{l}\equiv m\pmod{p^{\gamma}}$$ is soluble for all $m\in\mathbb{N}.$ 
The result follows then using the previous remark and observing that under the conditions described above, the same lemma delivers a representation
$$\sum_{i=1}^{3}y_{i}^{k}+\sum_{i=1}^{s}z_{i}^{k}\equiv n\pmod{p^{\gamma}}$$with $p\nmid y_{1}.$
\end{proof}
\begin{prop}\label{prop3}
Let $s\geq \max(1,k-2).$ Then one has that \begin{equation}\label{produ}\mathfrak{S}(n)=\prod_{p}\sigma(p),\end{equation} the singular series $\mathfrak{S}(n)$ converges and $\mathfrak{S}(n)\ll 1.$ Also, for $Q>0$ we obtain the estimate
\begin{equation}\label{qkes}\sum_{q=1}^{Q}q^{1/k}\lvert S_{n}(q)\rvert\ll Q^{\varepsilon}.\end{equation}
Moreover, if $s$ satisfies the conditions of Lemma \ref{lem4.3} then $\mathfrak{S}(n)\gg 1$.
\end{prop}
\begin{proof}
The application of (\ref{putill}), (\ref{piu}) and (\ref{varp}) delivers $S_{n}(p^{h})\ll p^{h}w_{k}(p^{h})^{s+4},$ whence combining such bound with the definition of $w_{k}(q)$ we obtain the estimates
\begin{equation}\label{piy}\displaystyle\sum_{h=1}^{\infty}\lvert S_{n}(p^{h})\rvert\ll p^{-3/2},\ \ \ \ \ \  \ \ \ \ \ \ \ \ \ \sum_{h=1}^{\infty}p^{h/k}\lvert S_{n}(p^{h})\rvert\ll p^{-1}.\end{equation}Therefore, by the multiplicative property of $S_{n}(q)$ we get (\ref{produ}) and $$\sum_{q=1}^{Q}\lvert S_{n}(q)\rvert\ll \prod_{p\leq Q}(1+Cp^{-3/2})\ll 1,$$
which delivers the upper bound for $\mathfrak{S}(n)$. The estimate (\ref{qkes}) follows in a similar way.

Observe that expressing $S_{n}(p^{h})$ as the difference of two complete exponential sums and using orthogonality we get
\begin{equation}\label{Snm}\sum_{j=0}^{h}S_{n}(p^{j})=M_{n}(p^{h})p^{-h(st+1)}\varphi(p^{h})^{-2}.\end{equation} We use Lemma \ref{lem4.3} and the fact that if $m$ with $(m,p)=1$ is a $k$-th power modulo $p^{\gamma}$ then it is also a $k$-th power modulo $p^{h}$ for $h\geq \gamma$ to obtain the lower bound $M_{n}(p^{h})\geq p^{(st+3)(h-\gamma)}$. Combining such lower bound with the above expression and (\ref{sigma}) we find that $\sigma(p)\geq p^{-\gamma(st+3)}.$ Therefore, by the preceding discussion and equation (\ref{piy}) we get $\frak{S}(n)\gg 1$.
\end{proof}

We next define an analogous singular series that arises in the analysis of the major arc contribution in Theorem \ref{thm1.3}. This series will be in nature quite similar to $\frak{S}(n)$, so we will skip some details for the sake of brevity. For such purposes, consider
\begin{equation}\label{Sn'}S'_{n}(q)=\sum_{\substack{a=1\\ (a,q)=1}}^{q}\big(q^{-t}S(q,a)\big)^{s}e_{q}(-an).\end{equation} 
Observe that using (\ref{piu}) then we have
\begin{equation}\label{appst}S'_{n}(p^{h})\ll p^{h}w_{k}(p^{h})^{s}.\end{equation}
Moreover, for any prime $p$ and $h\in\mathbb{N}$ and combining (\ref{sing}), (\ref{ec3.2}), (\ref{putill}) and (\ref{piu}) we can rewrite $S'_{n}(p^{h})$ as
\begin{equation}\label{apps}S_{n}'(p^{h})=\sum_{\substack{a=1\\ (a,p)=1}}^{p^{h}}\big(p^{-h}S_{k}(p^{h},a)\big)^{s}e_{p^{h}}(-an)+E(p^{h}),\end{equation}
where the first term is the analogous sum for the original Waring's problem and the error term satisfies $$E(p^{h})\ll p^{h-h/k-(t/l-1)+\varepsilon} w_{k}(p^{h})^{s-1}.$$ 

We define next the aforementioned series
\begin{equation*}\mathfrak{S}'(n)=\sum_{q=0}^{\infty}S'_{n}(q),\ \ \ \ \  \ \ \ \ \ \ \sigma'(p)=\displaystyle\sum_{h=0}^{\infty} S'_{n}(p^{h}),\end{equation*} where $p$ is a prime number. We can provide a more arithmetic description of $\sigma'(p)$ by considering the set 
$$\mathcal{M}^{*}_{n}(p^{h})=\Big\{\mathbf{X}\in [1,p^{h}]^{st}: \ p\nmid x_{1,1},\ p\nmid T(\mathbf{x}_{1}),\ \ n\equiv \sum_{i=1}^{s}T(\bfx_{i})^{k}\pmod{p^{h}}\Big\},$$ where $\mathbf{x}_{1}=(x_{1,1},\ldots, x_{1,t})$, and taking the counting function $M_{n}^{*}(p^{h})=\lvert\mathcal{M}_{n}^{*}(p^{h})\rvert$. For each prime $p$ take $\tau_{1}\geq 0$ such that $p^{\tau_{1}}\| kl$ and $\nu=\nu(p)=2\tau_{1}+1.$ Before stating the following lemma, recall (\ref{gama}).

\begin{lem}\label{lem4.45}
Suppose that $s\geq \frac{p}{p-1}\big(k,p^{\tau}(p-1)\big)$ when $\gamma=\tau+1$, that $s\geq 2^{\tau+2}$ when $\gamma=\tau+2$ and $k>2$, and that $s\geq 5$ when $p=k=2.$ Suppose as well that $t\geq 4l.$ Then one has $M_{n}^{*}(p^{\nu})>0.$
\end{lem}
\begin{proof}
It is worth noting first that since $\nu \geq \gamma$ then Lemma 2.15 of Vaughan \cite{Vau} and the fact that if $b\in\mathbb{N}$ with $(b,p)=1$ is a $k$-th power modulo $p^{\gamma}$ then it is also a $k$-th power modulo $p^{\nu}$ imply that $$T(\mathbf{x})=x_{1}^{l}+\dots+x_{t}^{l}\equiv m\pmod{p^{\nu}}$$ with $p\nmid x_{1}$ is soluble for all $m\in\mathbb{N}.$ 
The lemma follows using the previous remarks and observing that under the conditions described above, Lemma 2.15 of Vaughan \cite{Vau} delivers a representation
$$\sum_{i=1}^{s}y_{i}^{k}\equiv n\pmod{p^{\nu}}$$with $p\nmid y_{1}.$
\end{proof}
\begin{prop}\label{Prop4}
Let $s\geq \max(5,k+2).$ Then one has that \begin{equation*}\label{prod}\mathfrak{S}'(n)=\prod_{p}\sigma(p),\end{equation*} the singular series $\mathfrak{S}'(n)$ converges and $\mathfrak{S}'(n)\ll 1.$ Also, for $Q>0$ we obtain the estimate
\begin{equation}\label{qkesqt}\sum_{q=1}^{Q}q^{1/k}\lvert S_{n}'(q)\rvert\ll Q^{\varepsilon}.\end{equation}
Moreover, if $s$ satisfies the conditions of Lemma \ref{lem4.45} then $\mathfrak{S}'(n)\gg 1$.
\end{prop}
\begin{proof}Equation (\ref{appst}) and the multiplicativity of $w_{k}(q)$ yield
\begin{equation}\label{piyi}\displaystyle\sum_{h=1}^{\infty}\lvert S_{n}'(p^{h})\rvert\ll p^{-3/2},\ \ \ \ \ \  \ \ \ \ \ \ \ \ \ \sum_{h=1}^{\infty}p^{h/k}\lvert S_{n}'(p^{h})\rvert\ll p^{-1},\end{equation} which imply the convergence, the upper bound for the singular series and (\ref{qkesqt}). For obtaining the lower bound for $\frak{S}'(n)$ we observe that Lemma \ref{lem4.45}, an application of Hensel's Lemma and the argument used to derive (\ref{Snm}) allow one to obtain  $\sigma'(p)\geq p^{-\nu(st-1)}$, whence combining such estimate with (\ref{piyi}) we then get $\frak{S}'(n)\gg 1$.
\end{proof}
Scholars in the area will realise that one could use (\ref{apps}) and the bounds available in the literature for the singular series in the original Waring's problem (see \cite[Theorem 4.5]{Vau}) to obtain $\mathfrak{S}'(n)\gg 1$ for the range $s\geq \max(4,k+1)$. One could also deduce (\ref{qkesqt}) but with an extra factor of $n^{\varepsilon}$ in the right side of the bound for the same range (see \cite[Lemma 4.8]{Vau}).
\section{Approximation of some exponential sum over the major arcs}\label{sec4}
In this section we approximate $f(\alpha)$ on the major arcs by some auxiliary function. For such purpose it is convenient to introduce first some notation. Let $\alpha\in [0,1)$ and $a\in\mathbb{Z}$, $q\in\mathbb{N}$ with $(a,q)=1$. Denote $\beta=\alpha-a/q$ and consider the aforementioned function $U(\alpha,q,a)=c_{t,l}q^{-t}S(q,a)u(\beta)$, where
\begin{equation}\label{Ualp}u(\beta)=k^{-1}\sum_{m=1}^{n}m^{t/kl-1}e(\beta m),\ \ \ \  \ \ \ \ c_{t,l}=\Gamma(1+1/l)^{t}\Gamma(t/l)^{-1},\end{equation} and $S(q,a)$ was defined in (\ref{S(q,a)}). 
\begin{prop}\label{prop200}
Let $q<P$ and $\alpha,a,q,\beta$ as above. Then one has
$$f(\alpha)=U(\alpha,q,a)+O\big(qP^{t-1}(1+n\lvert\beta\rvert)\big).$$
\end{prop}
\begin{proof}
Before embarking on our task, it is convenient to consider for each $\mathbf{r}\in\mathbb{N}^{t}$ and $r\in\mathbb{N}$ the sums
 \begin{equation*}K_{\mathbf{r}}(\beta)=\displaystyle\sum_{\substack{T(\mathbf{x})\leq P^{l}\\ \mathbf{x}\equiv\mathbf{r}\pmod{q}}}e\big(\beta T(\mathbf{x})^{k}\big),\ \ \ \ \ \ \ B_{r}(x)=\displaystyle\sum_{\substack{0<z\leq x\\ z\equiv r\pmod{q}}}1.\end{equation*} Observe that by sorting the summation into arithmetic progressions modulo $q$ we find that
\begin{equation}\label{halpa}
f(\alpha)=\sum_{\mathbf{r}\leq q}e_{q}\big(aT(\mathbf{r})^{k}\big)K_{\mathbf{r}}(\beta).
\end{equation} 
For each $\mathbf{r}\in\mathbb{N}^{t}$ write $\mathbf{r}=(\mathbf{r}_{t-1},r_{t}),$ where $\mathbf{r}_{t-1}\in\mathbb{N}^{t-1}$. Then recalling the definition of $P_{\mathbf{x}}$ after (\ref{f(alpha)}) and using Abel's summation formula we find that
\begin{align*}
K_{\mathbf{r}}(\beta)=&
\sum_{\mathbf{x}}B_{r_{t}}(P_{\mathbf{x}})e\big(\beta T(\mathbf{x},P_{\mathbf{x}})^{k}\big)-\sum_{\mathbf{x}}\int_{0}^{P_{\mathbf{x}}}\frac{\partial }{\partial z}e\big(\beta T(\mathbf{x},z)^{k}\big)B_{r_{t}}(z)dz,
\end{align*}
where $\mathbf{x}\in\mathbb{N}^{t-1}$ runs over tuples satisfying $T_{t-1}(\mathbf{x})\leq P^{l}-1$ and $\mathbf{x}\equiv \mathbf{r}_{t-1}\mmod{q}$. Consequently, combining the formula $B_{r_{t}}(x)=x/q+O(1)$ and an application of integration by parts one gets
\begin{equation*}K_{\mathbf{r}}(\beta)= q^{-1}\sum_{\mathbf{x}}\int_{0}^{P_{\mathbf{x}}}e\big(\beta T(\mathbf{x},z)^{k}\big)dz+O\big(q^{-t+1}P^{t-1}(1+n\lvert\beta\rvert)\big).\end{equation*} We have included a brief discussion of the next step in the argument since it involves a technical detail which was not required before. Note first that Abel's summation formula combined with the above equation yields
\begin{align*}K_{\mathbf{r}}(\beta)= q^{-1}\sum_{\mathbf{x}}B_{r_{t-1}}(P_{(\mathbf{x},0)})I(P_{(\mathbf{x},0)})
&-q^{-1}\sum_{\mathbf{x}}\int_{0}^{P_{(\mathbf{x},0)}}\frac{\partial I}{\partial z_{t-1}}(z_{t-1})B_{r_{t-1}}(z_{t-1})dz_{t-1}
\\
&+O\big(q^{-t+1}P^{t-1}(1+n\lvert\beta\rvert)\big),
\end{align*}
where $\mathbf{x}\in\mathbb{N}^{t-2}$ runs over tuples satisfying $T_{t-2}(\mathbf{x})\leq P^{l}$ with $\mathbf{x}\equiv \mathbf{r}_{t-2}\mmod{q}$ and $$I(z_{t-1})=\int_{0}^{P_{\mathbf{x},z_{t-1}}}e\big(\beta T(\mathbf{x},z_{t-1},z_{t})^{k}\big)dz_{t}.$$ Observe that combining the Fundamental Theorem of Calculus and the exchangeability of derivation and integration we find that $$\frac{\partial}{\partial z_{t-1}}I(z_{t-1})\ll \frac{\partial}{\partial z_{t-1}}P_{\mathbf{x},z_{t-1}}+n\lvert\beta\rvert,$$whence another combination of the formula $B_{r_{t-1}}(x)=x/q+O(1)$ and integration by parts yields
\begin{equation*}K_{\mathbf{r}}(\beta)= q^{-2}\sum_{\mathbf{x}}\int_{\mathcal{C}}e\big(\beta T(\mathbf{x},z_{t-1},z_{t})^{k}\big)dz_{t-1}dz_{t}+O\big(q^{-t+1}P^{t-1}(1+n\lvert\beta\rvert)\big),\end{equation*}
where $\mathcal{C}$ is the set of pairs $(z_{t-1},z_{t})\in\mathbb{R}_{+}^{2}$ satisfying $T(\mathbf{x},z_{t-1},z_{t})\leq P^{l}$. We repeat a similar argument for the rest of the variables to obtain
\begin{equation}\label{krb}K_{\mathbf{r}}(\beta)=q^{-t}u_{1}(\beta)+O\big(q^{-t+1}P^{t-1}(1+n\lvert\beta\rvert)\big),\end{equation} where $u_{1}(\beta)$ here denotes the integral version of $u(\beta)$, which we define by
$$u_{1}(\beta)=\int_{T(\mathbf{x})\leq P^{l}}e\big(\beta T(\mathbf{x})^{k}\big)d\mathbf{x},$$
where $\mathbf{x}\in\mathbb{R}_{+}^{t}.$ Observe that by several changes of variables, one can rewrite $u_{1}(\beta)$ as
$$u_{1}(\beta)=k^{-1}l^{-t}\int_{0}^{n}w^{1/k-1}e(\beta w)\int_{\mathbf{X}\in \mathcal{M}}B(w,\mathbf{X})d\mathbf{X}dw,$$ with $B(w,\mathbf{X})=x_{1}^{1/l-1}\cdots x_{t-1}^{1/l-1}(w^{1/k}-x_{1}-\ldots-x_{t-1})^{1/l-1}$ and $\mathcal{M}\subset \mathbb{R}^{t-1}$ is the corresponding set determined by the underlying inequalities. Consequently, combining the formula for the Euler-Beta function and subsequent changes of variables we get
$$u_{1}(\beta)=c_{t,l}u_{2}(\beta),\ \ \ \ \ \text{where}\ u_{2}(\beta)=k^{-1}\int_{0}^{n}w^{t/kl-1}e(\beta w)dw.$$

We devote the rest of the proof to compute the error term when we approximate $u_{2}(\beta)$ by $u(\beta)$. We believe that working with $u(\beta)$ instead makes the analysis a bit more transparent and less tedious. Consider the function $$G(\gamma)=\sum_{1\leq y\leq \gamma}y^{t/kl-1}$$ and note that the Euler-Maclaurin formula (see Vaughan \cite[(4.8)]{Vau}) yields 
$$G(\gamma)=klt^{-1}\gamma^{t/kl}+O\big(Z(\gamma)\big),$$ where $Z(\gamma)=1+\gamma^{t/kl-1}.$ Observe that then Abel's summation formula, integration by parts and the previous discussion yield 
\begin{align*}u(\beta)&
=lt^{-1}n^{t/kl}e(\beta n)-2\pi i \beta\int_{0}^{n}lt^{-1}\gamma^{t/kl}e(\beta\gamma)d\gamma+O\big(Z(n)(1+n\lvert\beta\rvert)\big)
\\
&=u_{2}(\beta)+O\big(Z(n)(1+n\lvert\beta\rvert)\big).
\end{align*}
The lemma then follows combining the above approximation with (\ref{halpa}) and (\ref{krb}).
\end{proof}
Note that the error term in the above proposition differs from the trivial bound by a factor of $Pq^{-1}(1+n\vert\beta\rvert)^{-1}$, and this saving is gained by fixing $t-1$ variables in the expression for $K_{\mathbf{r}}(\beta)$ and using a one-dimensional argument. The saving obtained in Proposition \ref{prop200}, however, is not enough for our choice of the major arcs when $l$ is large enough. Likewise, the possible approaches involving the use of all of the variables don't seem to improve the error term substantially. We devote the rest of the section to provide an upper bound for $u(\beta).$

\begin{lem}\label{lem4.1}
Let $\lvert\beta\rvert\leq 1/2.$ Then one has
$$u(\beta)\ll \frac{P^{t}}{(1+n\lvert\beta\rvert)^{\gamma_{k,l}}},$$
where $\gamma_{k,l}=\min(1,t/kl).$
\end{lem}
\begin{proof}
When $\lvert\beta\rvert\leq n^{-1}$ one finds that
$$u(\beta)\ll \sum_{m=1}^{n}m^{t/kl-1}\ll P^{t},$$ which yields the required bound for that particular range. When $\lvert\beta\rvert> n^{-1}$ then denoting $M=\lfloor\lvert\beta\rvert^{-1}\rfloor$, we observe that $$\sum_{m=1}^{M}m^{t/kl-1}e(\beta m)\ll \lvert\beta\rvert^{-t/kl}.$$ For the remaining range we combine partial summation and the monotonicity of $m^{t/kl-1}$ to obtain 
$$\sum_{m>M}^{n}m^{t/kl-1}e(\beta m)\ll \lvert\beta\rvert^{-1}\big(\lvert\beta\rvert^{1-t/kl}+n^{t/kl-1}\big)=\lvert\beta\rvert^{-t/kl}+P^{t}\lvert\beta\rvert^{-1}n^{-1},$$ which delivers the required estimate.
\end{proof}

\section{Treatment of the major arcs and proof of the main theorem}\label{sec5}
In this section we prune back to the narrower set $\grP$ of major arcs and deduce an asymptotic formula for the contribution of such set. In view of the weak bound obtained in Proposition \ref{prop200} and the discussion made after it we are forced to introduce $k$-th powers of natural numbers and prime numbers, whose behaviour is much better understood, to provide strong enough estimates over $\grM$. For such purposes, it is convenient to present first some notation. Let \begin{equation*}X_{1}=2^{-1}(2k)^{-1/(k-1)}X.\end{equation*}Consider the exponential sums
$$g(\alpha)=\sum_{X_{1}< x\leq 2X_{1}}e(\alpha x^{k}),\ \ \ \ \ \ \ \ \ \ \  \ \ h(\alpha)=\sum_{p\leq X}e(\alpha p^{k}),$$
the weighted sums
\begin{equation*}v(\beta)=\sum_{X_{1}^{k}< x\leq (2X_{1})^{k}}k^{-1}x^{1/k-1}e(\beta x),\ \ \ \ \ \ \ w(\beta)=\sum_{2\leq x\leq n}k^{-1}x^{1/k-1}(\log x)^{-1}e(\beta x),\end{equation*} and the functions
$$V(\alpha,q,a)=q^{-1}S_{k}(q,a)v(\beta)\ \ \ \ \ \ \ \ \ \text{and}\ \ \ \ \ \ \ \ \  W(\alpha,q,a)=\varphi(q)^{-1}W(q,a)w(\beta),$$ where $S_{k}(q,a)$ and $W(q,a)$ were defined in (\ref{S(q,a)}) and (\ref{WW}) respectively. For the sake of simplicity we further define the auxiliary functions \begin{equation}\label{f*}f^{*}(\alpha)=U(\alpha,q,a),\ \ \ \ \ \ \ g^{*}(\alpha)=V(\alpha,q,a),\ \ \ \ \ \ \ h^{*}(\alpha)=W(\alpha,q,a)\end{equation} when $\alpha\in\grM(a,q)\subset \grM$ and $f^{*}(\alpha)=g^{*}(\alpha)=h^{*}(\alpha)=0$ for $\alpha\in\grm.$ We recall for convenience that $U(\alpha,q,a)$ was defined just before (\ref{Ualp}). Before providing an asymptotic formula for the major arc contribution it is convenient to consider for any set $\frak{B}\subset [0,1)$ the integral
$$R_{\frak{B}}(n)=\int_{\frak{B}}f(\alpha)^{s}g(\alpha)^{2}h(\alpha)^{2}e(-\alpha n)d\alpha,$$ and to define the singular integral as
$$J(n)=\int_{0}^{1}u(\beta)^{s}v(\beta)^{2}w(\beta)^{2}e(-\beta n)d\beta.$$ Here the reader might want to observe that as a consequence of orthogonality then $J(n)$ equals
$$\sum_{x_{1},\dots,x_{4},y_{1},\dots,y_{s}}k^{-4-s}(x_{1}x_{2}x_{3}x_{4})^{1/k-1}\big(\log x_{3}\log x_{4}\big)^{-1}y_{1}^{t/kl-1}\cdots y_{s}^{t/kl-1},$$
where the sum is over $x_{1},\ldots,x_{4},y_{1},\ldots,y_{s}$ satisfying $x_{1}+\ldots+x_{4}+y_{1}+\ldots+y_{s}=n$ with
$$X_{1}^{k}<x_{1},x_{2}\leq (2X_{1})^{k},\ \ \ \ \ \ \ 2\leq x_{3},x_{4}\leq n\ \ \ \ \ \ \text{and}\ \ \ \ \ \ \ 1\leq y_{j}\leq n\ \ \ \ \ (1\leq j\leq s).$$
It is worth observing that then one obtains the lower bound \begin{equation}\label{JNN}J(n)\gg P^{st}X^{4}n^{-1}(\log n)^{-2}.\end{equation}
\begin{prop}\label{prop55} Let $s\geq \max(1,k-2).$ One has that
\begin{equation}\label{Rma}R_{\grM}(n)=\frak{S}(n)J(n)+O\big(P^{st}X^{4}n^{-1}(\log n)^{-2-\delta}\big).\end{equation} Moreover, if $s$ satisfies the hypothesis of Lemma \ref{lem4.3} then $R_{\grM}(n)\gg P^{st}X^{4}n^{-1}(\log n)^{-2}.$
\end{prop}
\begin{proof}
Observe that Lemma 6.1 of Vaughan \cite{Vau} for the choice of $X_{1}$ made yields that whenever $\alpha\in\grM$ then
\begin{equation}\label{pro}g(\alpha)-g^{*}(\alpha)\ll q^{1/2+\varepsilon}.\end{equation} 
Likewise, Lemma 6.2 of Vaughan \cite{Vau} delivers the bound $v(\beta)\ll X(1+n\lvert\beta\rvert)^{-1}$, whence combining such estimate with (\ref{putill}) we get \begin{equation}\label{ges} g^{*}(\alpha)\ll w_{k}(q)X(1+n\lvert\beta\rvert)^{-1}.\end{equation}
It is also worth noting that for any $q\leq X$, the number $N(q)$ of pairs of primes $(p_{1},p_{2})$ with $p_{1}^{k}\equiv p_{2}^{k}\pmod{q}$ and $p_{1},p_{2}\leq X$ satisfies $N(q)\ll X^{2}(\log X)^{-2}q^{-1+\varepsilon}.$ Consequently, by orthogonality we find that
\begin{equation}\label{orti}\sum_{a=1}^{q}\lvert h(a/q+\beta)\rvert^{2}\ll X^{2}(\log X)^{-2}q^{\varepsilon}.\end{equation}
Combining the previous discussion with Lemma \ref{lem2} we obtain
\begin{align*}R_{\grM\setminus \grN}(n)\ll P^{st-2^{-lk}+\varepsilon}X^{4}n^{-1}\big(1+\sum_{q\leq X}w_{k}(q)^{2}\big)\ll P^{st-\delta}X^{4}n^{-1},
\end{align*}
where in the last step we applied (\ref{pri}). The reader might want to observe that the bounds available for the exponential sums of $k$-th powers are robust enough to enable us to prune back to a set of narrower major arcs. It becomes transparent that in view of the weak estimates for $f(\alpha)$ available when $\alpha\in\grM\setminus \grN$, the use of such Weyl sums in this setting seems inevitable. Before moving on, it is convenient to observe that whenever $\alpha\in\grN$ then equations (\ref{piu}), (\ref{Ualp}) and Proposition \ref{prop200} deliver $f(\alpha)\ll w_{k}(q)P^{t}$. Likewise, observe that (\ref{pro}) and (\ref{ges}) yield the estimate $g(\alpha)\ll w_{k}(q)X(1+n\lvert\beta\rvert)^{-1}$ for the same range. Consequently, combining the previous discussion with (\ref{pri}) and (\ref{orti}) we obtain
\begin{align*}R_{\grN\setminus\grP}(n)\nonumber
&\ll P^{st}X^{4}n^{-1}(\log P)^{-2}\big((\log P)^{-1}\sum_{q\leq \log P}q^{1+\varepsilon}w_{k}(q)^{s+2}+\sum_{q>\log P}w_{k}(q)^{s+2}\big)
\\
&\ll P^{st}X^{4}n^{-1}(\log P)^{-2-\delta}.
\end{align*}

In order to make further progress in the analysis, we note that 
\begin{equation*}\label{hes}h(\alpha)=W(\alpha,q,a)+O(Xe^{-C_{1}\sqrt{\log X}})\end{equation*} for some $C_{1}>0$, which is an immediate consequence of Lemma 7.15 of Hua \cite{Hua2}. Observe as well that Proposition \ref{prop200} delivers $f(\alpha)^{s}-f^{*}(\alpha)^{s}\ll P^{st-1+\varepsilon}$ whenever $\alpha\in\grP$. Combining these estimates with (\ref{pro}) we find that
$$\int_{\grP} \big\lvert f(\alpha)^{s} g(\alpha)^{2}h(\alpha)^{2}-f^{*}(\alpha)^{s} g^{*}(\alpha)^{2}h^{*}(\alpha)^{2}\big\rvert d\alpha\ll P^{st}X^{4}n^{-1}e^{-C\sqrt{\log X}}.$$
Observe as well that (\ref{qkes}) and the estimate for $v(\beta)$ stated before (\ref{ges}) deliver the bounds
$$\sum_{q>Q}\lvert S_{n}(q)\rvert\ll Q^{\varepsilon-1/k},\ \ \ \ \ \int_{\lvert\beta\rvert>(\log P)q^{-1}n^{-1}}\lvert v(\beta)\rvert^{2} d\beta\ll X^{2}n^{-1}q(\log P)^{-1}$$ for any $Q>0$ and $q\leq \log P$ respectively. Consequently, the above estimates and a change of variables yield
$$\int_{\grP}f^{*}(\alpha)^{s}g^{*}(\alpha)^{2}h^{*}(\alpha)^{2}e(-\alpha n)d\alpha=\frak{S}(n)J(n)+O\big(P^{st}X^{4}n^{-1}(\log P)^{-2-\delta}\big),$$whence the preceding discussion and the pruning bounds for $R_{\grM\setminus \grN}(n)$ and $R_{\grN\setminus\grP}(n)$ deliver the main result of the proposition. The second part of the proposition follows combining  (\ref{JNN}) with (\ref{Rma}) and Proposition \ref{prop3}.
\end{proof}
We now gather all the work done previously to prove Theorem \ref{thm1.2} by using the following quantitative version.
\begin{prop}\label{propthm}
Let $s\geq 4k-3$ and $H=k(k+1).$ Then, one has the lower bound
$$\int_{0}^{1} \mathcal{F}(\alpha)^{H}f(\alpha)^{s}g(\alpha)^{2}h(\alpha)^{2}e(-\alpha n)d\alpha\gg |\mathcal{S}_{1}|^{H}|\mathcal{S}_{2}|^{H}P^{st}X^{4}n^{-1}(\log n)^{-2}.$$
\end{prop}
\begin{proof}
It is worth observing that the estimate over the minor arcs on Proposition \ref{prop22} and the trivial bounds for $f(\alpha),$ $g(\alpha)$ and $h(\alpha)$ yield
\begin{equation}\label{min}\int_{\grm} \lvert\mathcal{F}(\alpha)\rvert^{H}\lvert f(\alpha)\rvert^{s}\lvert g(\alpha)\rvert^{2}\lvert h(\alpha)\rvert^{2}d\alpha\ll |\mathcal{S}_{1}|^{H}|\mathcal{S}_{2}|^{H}P^{st}X^{4}n^{-1-\delta}\end{equation} for some $\delta>0$. In order to compute the major arc contribution it is convenient to define for each $m\in\mathbb{N}$ the counting function
$$Q(m)=\Big\lvert \Big\{(\mathbf{x},\mathbf{y})\in \mathcal{S}_{1}^{H}\times \mathcal{S}_{2}^{H} :\ \ \ m=\sum_{i=1}^{H}(y_{i}+z_{i})^{k}\Big\}\Big\rvert.$$Observe that with the previous notation one finds that
$$\mathcal{F}(\alpha)^{H}=\sum_{m}Q(m)e(\alpha m).$$ Moreover, using (\ref{C1}), (\ref{Pt}) and the definitions of $\mathcal{S}_{1}$ and $\mathcal{S}_{2}$ described after those equations we have $Q(m)=0$ for $m>n/2.$ Therefore, Proposition \ref{prop55} yields
\begin{align}\label{Ma}\int_{\grM} \mathcal{F}(\alpha)^{H}f(\alpha)^{s}g(\alpha)^{2}h(\alpha)^{2}e(-\alpha n)d\alpha
&=\sum_{m\leq n/2}Q(m)R_{\grM}(n-m)\nonumber
\\
&\gg |\mathcal{S}_{1}|^{H}|\mathcal{S}_{2}|^{H} P^{st}X^{4}n^{-1}(\log n)^{-2}.
\end{align}
The combination of the equations (\ref{min}) and (\ref{Ma}) concludes the proof. Here the reader might want to observe that the choices for $C_{1}$ and $C_{2}$ guarantee that we get the expected lower bound. 
\end{proof}

\emph{Proof of Theorem \ref{thm1.2}}. Note that the integral in Proposition \ref{propthm} counts the number of solutions of equation (\ref{conje}) with certain multiplicities. Consequently, for all $r\geq 4$ we have\begin{equation*}\label{pytr}S_{C}(k,l,r)\leq k(k+1)+4k-3,\end{equation*}
which delivers the same bound for $P_{C}(k,r)$ and yields $R_{C}(k)\leq 4$. As experts will realise, we could have avoided including the extra $4k-3$ copies of $f(\alpha)$ by introducing suitable weights for each of $x$ and $m$ in the definition of $\mathcal{F}(\alpha)$ to exploit the information given by such variables in the analysis of the singular series. However, we have prioritised the simplicity of the exposition over the preciseness of the upper bound for $P_{C}(k,r)$.
\section{The case $k=2$}\label{sec6}
We briefly sketch the proof for $R_{4}(2)\leq 2$. For the rest of the exposition then we take $t=4l$. Let $$g(\alpha)=\sum_{X/2<x\leq X}e(\alpha x^{2}),$$and on recalling (\ref{falphap}) consider the mean value
$$\int_{0}^{1}\lvert g(\alpha)\rvert^{2} \lvert f(\alpha,\mathcal{S}_{t})\rvert^{4}d\alpha,$$which by orthogonality counts the solutions to the equation $$x_{1}^{2}+y_{1}^{2}+y_{2}^{2}=x_{2}^{2}+y_{3}^{2}+y_{4}^{2}$$with $X/2\leq x_{i}\leq X$ and $y_{i}\in\mathcal{S}_{t}$. Observe that by a divisor function argument, the number of solutions of $$y_{1}^{2}+y_{2}^{2}=y_{3}^{2}+y_{4}^{2}$$ is $O(n^{\varepsilon}\lvert\mathcal{S}_{t}\rvert^{2})$, and hence the contribution of the subset of solutions satisfying $x_{1}=x_{2}$ is $O(n^{1/2+\varepsilon}\lvert\mathcal{S}_{t}\rvert^{2})$. Likewise, the number of solutions of the above equation with $x_{1}\neq x_{2}$ is $O(n^{\varepsilon}\lvert \mathcal{S}_{t}\rvert^{4})$, whence equation (\ref{1.1}) and the above estimates deliver
$$\int_{0}^{1}\lvert g(\alpha)\rvert^{2} \lvert f(\alpha,\mathcal{S}_{t})\rvert^{4}d\alpha\ll n^{1/2+\varepsilon}\lvert\mathcal{S}_{t}\rvert^{2}+n^{\varepsilon}\lvert \mathcal{S}_{t}\rvert^{4}\ll n^{\varepsilon}\lvert \mathcal{S}_{t}\rvert^{4}.$$ Define $\grM_{\tau}=\grM(P^{\tau})$ and $\grm_{\tau}=[0,1)\setminus \grM_{\tau}$ for some small enough $\tau>0$. Then, combining the above estimate with Lemma \ref{lem2} one gets
$$\int_{\grm_{\tau}}\lvert f(\alpha,\mathcal{S}_{t})\rvert^{4}\lvert g(\alpha)\rvert^{2} \lvert f(\alpha)\rvert^{3}d\alpha\ll \lvert \mathcal{S}_{t}\rvert^{4}P^{3t-\delta}.$$ 
The reader might want to observe that in order to ensure local solubility and the convergence of the singular series, one should take $3$ copies of $f(\alpha)$ instead of just $2$ since we only have two Weyl sums of degree $2$ available. The rest of the analysis of the major arcs is done using the estimates obtained throughout the memoir. This argument then yields $P_{4}(2,2)\leq 7$. As experts will realise, one could prove the bound $P_{4}(2,2)\leq 5$ by introducing suitable weights in the definition of $f(\alpha,\mathcal{S}_{t})$ to simplify the singular series analysis and just use one copy of $f(\alpha)$. We have avoided discussing such refinement here for the sake of brevity.
\section{Proof of Theorem \ref{thm1.3}}\label{sec7}
We combine the work of previous sections with slightly different ideas employed in the minor arc analysis to give a proof of Theorem \ref{thm1.3}. We first introduce an exponential sum restricted to elements in a convenient set that will provide the saving needed for the minor arc contribution. On recalling the parameter $\xi_{0}(k,l)$ defined before Theorem \ref{thm1.3}, we write $\xi=\xi_{0}(k,l)$ for ease of notation and consider $$\xi_{1}=\xi-1,\ \ \ \ \ \ C_{3}=\big(2^{k+1}k(k+1)\big)^{-1/kl}\xi_{1}^{-1/l}.$$ Recalling (\ref{Sr}) as well, take $P_{3}=C_{3}P$ and set $\mathcal{S}=\mathcal{S}_{\xi_{1}}(P_{3})$. It is then convenient to consider for $P/2\leq p\leq P$ prime the exponential sums
\begin{equation*}f_{p}(\alpha)=\sum_{x\in\mathcal{S}}e\big(\alpha(x+p^{l})^{k}\big)\ \ \ \ \ \ \ \  \text{and}\ \ \ \ \ \ \ \mathcal{G}(\alpha)=\sum_{P/2<p\leq P}f_{p}(\alpha).\end{equation*}

\begin{prop}\label{prop223}
Let $\alpha\in[0,1)$ and let $M>0$ be a parameter with $M\leq P.$ Denote by $\grm_{M}$ the set of $\alpha$ with the property that $\alpha=\beta+a/q$ with $a\in\mathbb{Z},$ $q\in\mathbb{N}$ and $(a,q)=1$ satisfying $\lvert\beta\rvert\leq (2kqX)^{-1}$, $q\leq 2kX$ and such that whenever $q\leq M$ one has $\lvert\beta\rvert\geq Mq^{-1}n^{-1}.$ Then for each $\alpha\in\grm_{M}$ one gets
$$\mathcal{G}(\alpha)\ll |\mathcal{S}|P^{1+\varepsilon}M^{-1/k(k-1)}X^{\delta_{\xi_{1}}/2},$$ where $\delta_{\xi_{1}}$ was defined in (\ref{cono}). Moreover, for $s\geq k(k+1)/2$ we obtain the mean value
\begin{equation*}\int_{\grm_{M}}\lvert\mathcal{G}(\alpha)\rvert^{2s}d\alpha\ll |\mathcal{S}|^{2s}P^{2s}M^{-1}X^{-k+\Delta_{\xi_{1}}+\varepsilon},\end{equation*}where $\Delta_{\xi_{1}}=\delta_{\xi_{1}}k(k+1)/2.$
\end{prop}
\begin{proof}
Recalling the notation used in the proof of Proposition \ref{prop22} we find that
\begin{equation*}\sum_{P/2< p\leq P}\lvert f_{p}(\alpha)\rvert^{2s}=\sum_{P/2< p\leq P}\Big\lvert\sum_{\bfn\in\mathcal{N}}a(\bfn)e\big(\bfn\cdot\gamma(p^{l})\big)\Big\rvert^{2},\end{equation*}
where $\mathcal{N}$ here denotes the set
$$\mathcal{X}(\bfn)=\Big\{\mathbf{x}\in\mathcal{S}^{s}:\ x_{1}^{i}+\ldots+x_{s}^{i}=n_{i}\ (1\leq i\leq k-1)\Big\}$$ and the coefficient $a(\bfn)$ is defined in the same way as in (\ref{a(n}). Let $q_{1}=q(q,k)^{-1}.$ Before going into the discussion for the spacing modulo $1$ of $\{\gamma(p^{l})\}_{p}$, the reader might find useful to observe that for fixed $h\in\mathbb{Z}$ with $(h,q_{1})=1$, the number of solutions $L$ of the congruence $$x^{l}\equiv h\mmod{q_{1}}$$ satisfies $L\ll q_{1}^{\varepsilon}.$ Therefore, we can partition the primes into $L$ classes $\mathcal{P}_{j}$ such that for any pair of distinct primes $p_{1},p_{2}\in\mathcal{P}_{j}$ with $p_{1}^{l}\equiv p_{2}^{l}\mmod{q_{1}}$ then $p_{1}\equiv p_{2}\mmod{q_{1}}.$

Next observe that by the choice of $\gamma(p^{l})$ made in (\ref{ec1.33}) we find that 
$$\|k\alpha(p_{1}^{l}-p_{2}^{l})\|=\|\gamma_{k-1}(p_{1}^{l})-\gamma_{k-1}(p_{2}^{l})\|.$$ Then using the hypothesis on $\lvert\beta\rvert$ described above we obtain $$\|k\alpha(p_{1}^{l}-p_{2}^{l})\|\geq \| ka(p_{1}^{l}-p_{2}^{l})/q\|-\frac{1}{2}q^{-1}\geq \frac{1}{2}q^{-1}$$ provided that $p_{1}\not\equiv p_{2}\mmod{q_{1}}.$

When $q_{1}>P$ one cannot have pairs of distinct primes $p_{1},p_{2}\in\mathcal{P}_{j}$ with $p_{1}\equiv p_{2}\mmod{q_{1}},$ whence we always have $\|\gamma_{k-1}(p_{1}^{l})-\gamma_{k-1}(p_{2}^{l})\|\gg X^{-1}.$ Whenever $M/k< q_{1}\leq P$ then we partition each of $\mathcal{P}_{j}$ into $L_{j}$ classes $\mathcal{P}_{j,i}$ with the property that no pair of distinct primes belonging to $\mathcal{P}_{j,i}$ are congruent modulo $q_{1}$ and with $L_{j}$ satisfying the bound $L_{j}\ll Pq_{1}^{-1}.$ Consequently, the same argument leads to the estimate $\|\gamma_{k-1}(p_{1}^{l})-\gamma_{k-1}(p_{2}^{l})\|\gg X^{-1}$ for distinct $p_{1},p_{2}\in \mathcal{P}_{j,i}.$ Finally, when $q_{1}\leq M/k$ one has $q\leq M$, whence whenever $p_{1}\equiv p_{2}\mmod{q_{1}}$ then the condition on $\beta$ described in the proposition yields
$$\|\gamma_{k-1}(p_{1}^{l})-\gamma_{k-1}(p_{2}^{l})\|=\|k\alpha(p_{1}^{l}-p_{2}^{l})\|=\lvert \beta\rvert \lvert k(p_{1}^{l}-p_{2}^{l})\rvert\gg P^{l-1}Mn^{-1}.$$

We combine Lemma 5.3 of Vaughan \cite{Vau} and the above discussion to obtain the upper bound
\begin{equation}\label{minst}\sum_{\substack{P/2<p\leq P\\ p\in\mathcal{P}_{j}}}\lvert f_{p}(\alpha)\rvert^{2s}\ll X^{k(k-1)/2}P(q+M)^{-1}\sum_{\bfn\in\mathcal{N}}|a(\bfn)|^{2}\end{equation} for $q$ in any of the ranges described above. Bounding the coefficients $a(\bfn)$ trivially one gets
\begin{equation*}\label{mins}\sum_{\substack{P/2<p\leq P\\ p\in\mathcal{P}_{j}}}\lvert f_{p}(\alpha)\rvert^{2s}\ll X^{k(k-1)/2}P(q+M)^{-1}J_{s,\xi_{1}}^{(k-1)}(P_{3}),\end{equation*}
where $J_{s,\xi_{1}}^{(k-1)}(P_{3})$ was defined in (\ref{ec2.9}). Then combining the above equation with an application of Cauchy's inequality we get
\begin{equation*}\label{Falpha}\lvert\mathcal{G}(\alpha)\rvert^{2s}\ll P^{2s-1+\varepsilon}\sum_{j}\sum_{\substack{P/2<p\leq P\\ p\in\mathcal{P}_{j}}}\lvert f_{p}(\alpha)\rvert^{2s} \ll P^{2s+\varepsilon}X^{k(k-1)/2}M^{-1}J_{s,\xi_{1}}^{(k-1)}(P_{3}),\end{equation*} and hence for the choice $s=k(k-1)/2$ then Proposition \ref{prop1} delivers
$$\mathcal{G}(\alpha)\ll |\mathcal{S}|P^{1+\varepsilon}M^{-1/k(k-1)}X^{\delta_{\xi_{1}}/2}.$$ 

For the second claim of the proposition we combine (\ref{minst}) and Cauchy's inequality in the same way as above and we integrate over $\grm_{M}$ to obtain
$$\int_{\grm_{M}}\lvert\mathcal{G}(\alpha)\rvert^{2s}d\alpha\ll P^{2s+\varepsilon}M^{-1}X^{k(k-1)/2}J_{s,\xi_{1}}^{(k)}(P_{3}).$$An application of Proposition \ref{prop1} to the above line then yields 
 $$\int_{\grm_{M}}\lvert\mathcal{G}(\alpha)\rvert^{2s}d\alpha\ll |\mathcal{S}|^{2s}P^{2s+\varepsilon}M^{-1}X^{-k+\Delta_{\xi_{1}}},$$from where the second statement follows.
\end{proof}
In the rest of the section we deliver a lower bound for the major arc contribution. We will work with the auxiliary functions $f(\alpha)$, $S_{n}'(q)$, $u(\beta)$ and $f^{*}(\alpha)$, defined in (\ref{fte}), (\ref{Sn'}), (\ref{Ualp}) and (\ref{f*}) respectively but replacing $\xi$ by $t$ whenever such parameters appear in any of the definitions. We have avoided making such distinction in the notation explicit for the sake of simplicity. For future purposes we define the singular integral as $$J_{s}'(n)=\int_{-1/2}^{1/2}u(\beta)^{s}e(-\beta n)d\beta.$$ 
\begin{lem}\label{lem7.1}
Suppose that $s\geq 2$. Then,
$$J_{s}'(n)=n^{s\xi/kl-1}\Big(k^{-s}\Gamma(\xi/kl)^{s}\Gamma(s\xi/kl)^{-1}+O\big(B(n)\big)\Big),$$
where $B(n)=n^{-1}+n^{-\xi/kl}.$
\end{lem}
\begin{proof}
We will proceed by induction. We consider for convenience the function $\phi(\gamma)=\gamma^{\xi/kl-1}(n-\gamma)^{\xi/kl-1}$. When $s=2$ then orthogonality yields
\begin{align*}J_{2}'(n)&
=k^{-2}\sum_{m=1}^{n-1}\phi(m)=k^{-2}\int_{0}^{n}\phi(\gamma)d\gamma+O\big(n^{2\xi/kl-1}B(n)\big)
\\
&=k^{-2}\Gamma(\xi/kl)^{2}\Gamma(2\xi/kl)^{-1}n^{2\xi/kl-1}+O\big(n^{2\xi/kl-1}B(n)\big),
\end{align*}
where we used the fact that $\phi(\gamma)$ has at most one stationary point on the interval $(0,n).$ By using the inductive hypothesis we obtain
\begin{align*}J_{s+1}'(n)&
=k^{-1}\sum_{m=1}^{n-1}m^{\xi/kl-1}J_{s}'(n-m)
\\
&=k^{-s-1}\Gamma(\xi/lk)^{s}\Gamma(s\xi/kl)^{-1}\sum_{m=1}^{n-1}m^{\xi/kl-1}(n-m)^{s\xi/kl-1}+O\big(n^{(s+1)\xi/kl-1}B(n)\big).
\end{align*}
Applying the same argument we used for the case $s=2$ we find that
$$\sum_{m=1}^{n-1}m^{\xi/kl-1}(n-m)^{s\xi/kl-1}=n^{(s+1)\xi/kl-1}\big(\lambda_{s}+O\big(B(n)\big)\big),$$ where $\lambda_{s}=\Gamma(s\xi/kl)\Gamma(\xi/kl)\Gamma\big((s+1)\xi/kl\big)^{-1},$ whence combining the above equations we obtain the desired result.
\end{proof}
In order to make further progress we consider the set of major arcs $\grN_{\iota}=\grM(P^{1/2+\iota})$, where $\grM$ was defined in (\ref{grM}) and where we take $\iota=1/1000.$ Likewise, we define the minor arcs $\grn_{\iota}=[0,1)\setminus \grN_{\iota}.$ Note that using equation (\ref{piu}) and Proposition \ref{prop200} we obtain for $\alpha\in\grN_{\iota}$ the bound $$f(\alpha)^{s}-f^{*}(\alpha)^{s}\ll P^{s\xi-s/2+s\iota}+P^{\xi-1/2+\iota}w_{k}(q)^{s-1}u(\beta)^{s-1}.$$  Consequently, whenever $s\geq k+2$ then Lemma \ref{lem4.1} gives
\begin{align*}\int_{\grN_{\iota}}\lvert f(\alpha)^{s}-f^{*}(\alpha)^{s}\rvert d\alpha&
\ll P^{s\xi-s/2+(s+2)\iota+1}n^{-1}+P^{s\xi-1/2+\iota}n^{-1}\sum_{q\leq P^{1/2+\iota}}qw_{k}(q)^{s-1}
\\
&\ll P^{s\xi-\delta}n^{-1},
\end{align*}
where in the last step we used (\ref{pri}). Observe as well that (\ref{qkesqt}) and Lemma \ref{lem4.1} deliver the bounds
$$\sum_{q>Q}\lvert S'_{n}(q)\rvert\ll Q^{\varepsilon-1/k},\ \ \ \ \ \ \ \ \ \int_{\lvert\beta\rvert>P^{1/2+\iota}q^{-1}n^{-1}}\lvert u(\beta)\rvert^{s}d\beta\ll P^{s\xi}n^{-1}q^{\delta}P^{-\delta(1/2+\iota)}$$ whenever $s\geq \max(5,k+2)$ for any $Q>0$ and $q\leq P^{1/2+\iota}$ respectively. Therefore, Lemma \ref{lem7.1}, the above estimates and a change of variables give
\begin{equation}\label{JNu}\int_{\grN_{\iota}}f(\alpha)^{s}e(-\alpha n)d\alpha=C_{k,l,\xi}n^{s\xi/kl-1}\frak{S}'(n)+O\big(n^{s\xi/kl-1-\delta}\big),\end{equation} where $C_{k,l,\xi}=k^{-s}c_{\xi,l}^{s}\Gamma(\xi/lk)^{s}\Gamma(s\xi/kl)^{-1}$ and $c_{\xi,l}$ was defined in (\ref{Ualp}). Observe that when $s\geq 4k$ then Proposition \ref{Prop4} yields the lower bound $\frak{S}'(n)\gg 1.$ Likewise, Proposition \ref{prop223} delivers
\begin{equation*}\int_{\grn_{\iota}}\lvert\mathcal{G}(\alpha)\rvert^{k(k+1)}d\alpha\ll |\mathcal{S}|^{k(k+1)}P^{k(k+1)-1/2-\iota}X^{-k+\Delta_{\xi_{1}}+\varepsilon}\ll |\mathcal{S}|^{k(k+1)}P^{k(k+1)}X^{-k-\delta},\end{equation*}whence using (\ref{JNu}) and the ideas of the proof of Proposition \ref{propthm} to derive a lower bound for the major arc contribution and combining such bound with the above minor arc estimate we obtain
$$\int_{0}^{1}\mathcal{G}(\alpha)^{k(k+1)}f(\alpha)^{s}e(-\alpha n)d\alpha\gg |\mathcal{S}|^{k(k+1)}P^{k(k+1)+s\xi}(\log P)^{-k(k+1)}n^{-1},$$which concludes the proof of Theorem \ref{thm1.3}.

\end{document}